\documentclass[11pt,twoside,a4paper]{article}

\usepackage{algorithm}
\usepackage{algpseudocode}
\usepackage{graphicx,url}
\usepackage{setspace}
\usepackage{enumerate,lineno,setspace,float}
\usepackage[small,compact]{titlesec}
\usepackage{amsmath,amsfonts,amssymb,amsthm}
\usepackage{geometry}
\usepackage{fancyhdr}
\usepackage{amssymb,amsmath}
\usepackage{latexsym}
\usepackage{graphicx, float}
\usepackage{url}
\newtheorem{Theorem}{Theorem}[section]

\newtheorem{Corollary}[Theorem]{Corollary}

\newtheorem{Observation}[Theorem]{Observation}

\usepackage{color}

\definecolor{Blue}{rgb}{0,0,1}
\definecolor{Red}{rgb}{1,0,0}

\long\def\delete#1{}

\input amssym.def
\input amssym.tex

\newcommand{\be}{\begin{equation}}
\newcommand{\ee}{\end{equation}}
\newcommand{\bea}{\begin{eqnarray}}
\newcommand{\eea}{\end{eqnarray}}
\newcommand{\bean}{\begin{eqnarray*}}
\newcommand{\eean}{\end{eqnarray*}}

\def\({\left(}
\def\){\right)}
\def\[{\left[}
\def\]{\right]}

\begin{document}

\title{Forbidden induced subgraphs in iterative higher order line graphs}
\author{\textbf{Aryan Sanghi} \\ Indian Institute of Technology, Kharagpur - 721302, West Bengal (India) \\ E-mail : aryansanghi2004@gmail.com \\ \\
\textbf{Devsi Bantva} \\ Lukhdhirji Engineering College, Morvi - 363 642, Gujarat (India) \\ E-mail : devsi.bantva@gmail.com \\ \\
\textbf{Sudebkumar Prasant Pal}\footnote{Corresponding author.} \\ Indian Institute of Technology, Kharagpur - 721302, West Bengal (India) \\ E-mail : sudebkumar@gmail.com}

\pagestyle{myheadings}
\markboth{\centerline{Aryan Sanghi, Devsi Bantva and Sudebkumar Prasant Pal}}{\centerline{Forbidden induced subgraphs in iterative higher order line graphs}}

\date{}
\openup 0.8\jot
\maketitle

\begin{abstract}
Let $G$ be a simple finite connected graph. The line graph $L(G)$ of graph $G$ is the graph whose vertices are the edges of $G$, where $ef \in E(L(G))$ when $e \cap f	\neq \emptyset$. Iteratively, the higher order line graphs are defined inductively as $L^1(G) = L(G)$ and $L^n(G) = L(L^{n-1}(G))$ for $n \geq 2$. In \cite{Beineke}, Beineke characterize line graphs in terms of nine forbidden subgraphs. Inspired by this result, in this paper, we characterize second order line graphs in terms of pure forbidden induced subgraphs. We also give a sufficient list of forbidden subgraphs for a graph $G$ such that $G$ is a higher order line graph. We characterize all order line graphs of graph $G$ with $\Delta(G) = 3$ and $4$. \\


\end{abstract}

\section{Introduction}\label{intro}
Let $G$ be a simple, finite, connected graph without loops and multiple edges. The line graph of $G$, denoted by $L(G)$, is the graph whose vertices are the edges of $G$, where $ef \in E(L(G))$ when $e \cap f \neq \emptyset$. Iteratively, the higher order line graph is defined as $L^1(G) = L(G)$ and $L^n(G) = L(L^{n-1}(G))$ for $n \geq 2$. It is known that a connected graph $G$ is isomorphic to its line graph if and only if $G$ is a cycle. Thus, $L^n(C_m) = C_m$ for all $n$ while for $K_{1,3}$, $L^n(K_{1,3}) = L(K_{1,3})$ for all $n$, but $L(K_{1,3}) \neq K_{1,3}$. For a path $P_n$, $L(P_n) = P_{n-1}$,  $L^{n-1}(P_n)$ is a vertex, and $L^{m}(P_n)$ does not exist if $m \geq n$. In this work, we often use the statement that $G$ is a line graph that means $G$ is a line graph of some graph $H$. In other words, $G$ is a line graph that means $G = L(H)$ for some graph $H$. If $L(H) = G$ then we say that graph $H$ is a preimage of a line graph $G$, written as $L^{-1}(G) = H$. A clique is a graph in which every two vertices are adjacent. A subgraph $H$ of graph $G$ is called an induced subgraph if two vertices of $H$ are adjacent in $G$ they are also adjacent in $H$. 

In \cite{Whitney}, Whitney gave the following result about graphs and their line graphs, popularly known as the Whitney graph isomorphism theorem.

\begin{Theorem}\cite{Whitney}\label{whitney} $L(G^1) \cong L(G^2)$ and if $G^1$ and $G^2$ are not a complete graph of three nodes and a complete bipartite graph $K_{1,3}$, respectively, then $G^1 \cong G^2$.
\end{Theorem}

It is clear by the Whitney isomorphism theorem that there are only two connected graphs, namely $K_{1,3}$ and $K_3$, have the same line graph which is $K_3$. Moreover, there is a one-to-one correspondence between the set of graphs excluding $K_3$ and $K_{1,3}$ and, its line graphs by line graph operation. %However, both graphs $K_{1,3}$ and $K_3$ in any graph $G$ as subgraphs produce different line graphs depending on their adjacency in $G$.

\begin{Observation}{\rm Observe that if graph $G$ is a $K_3$ only then it is difficult to say which one from $K_{1,3}$ or $K_3$ is a preimage, but if $G \neq K_3$ is a connected line graph containing $K_3$ then its preimage can be easily determined as other vertex $v$ adjacent to a vertex of $K_{1,3}$ or $K_3$ makes different line graphs (see the following Figure \ref{Triangle}). This property plays a key role in finding the preimage of any line graph.}
\end{Observation}
\begin{figure}[ht!]
\centering
\includegraphics[width=11cm]{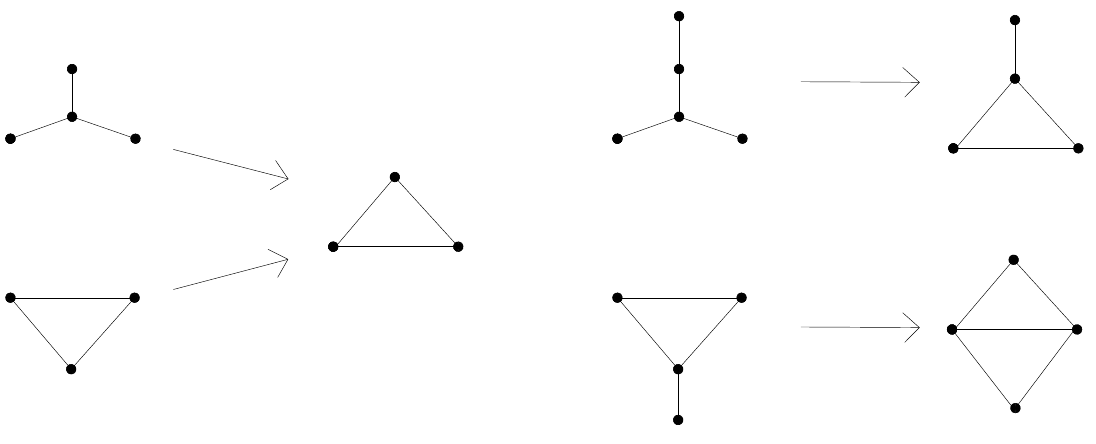}
\caption{Graphs $K_{1,3}$, $K_3$ and its line graphs.}
\label{Triangle}
\end{figure}

The following characterization of line graphs was given by Krausz in \cite{Krausz} and Rooij in \cite{Rooij}, respectively.

\begin{Theorem}\label{Krausz}\cite{Krausz} A graph is a line graph if and only if its edges can be partitioned into complete subgraphs in such a way that no vertex lies in more than two of these subgraphs.
\end{Theorem}

The partition of edges of a line graph into complete subgraphs as in Theorem \ref{Krausz} is called `line partition' and the partition sets (complete subgraphs of a line partition), denoted by $C^1,C^2,\ldots,C^k$, are known as 'cells' of line partition. A triangle of a graph $G$ is called an {\it odd} triangle if there is a vertex of $G$ adjacent to an odd number of its vertices (that is, adjacent to either one or all three vertices), and an {\it even triangle}, otherwise (that is, any vertex of graph $G$ that is adjacent to vertices of the triangle is adjacent to exactly two vertices of the triangle).

The following succinct characterization for line graphs states such graphs as having a $K_4$ whenever there are two distinct odd triangles sharing a common edge.

\begin{Theorem}\cite{Rooij}\label{rooij} A graph is a line graph if and only if it does not have the star $K_{1,3}$ as an induced subgraph and whenever $a,b,c$ and $b,c,d$ are distinct odd triangles, then $a$ and $d$ are adjacent.
\label{succinct-linegraph}
\end{Theorem}

\begin{Theorem}\cite{Rooij}\label{line:odd} If $G$ contains two even triangles with a common edge then it is one of the graphs $L(K_{1,3}+x) = E_1$, $L(E_1) = E_2$ or $L(K_4) = E_3$ (shown in following Figure \ref{Fig3}).
\begin{figure}[ht!]
\centering
\includegraphics[width=9cm]{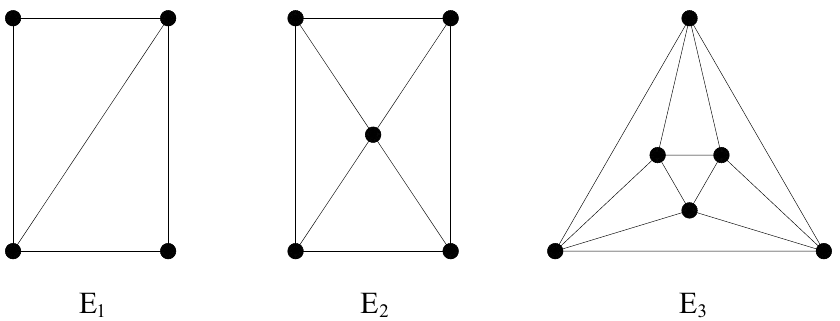}
\caption{Graphs $E_1$, $E_2$ and $E_3$.}
\label{Fig3}
\end{figure}
\end{Theorem}

\begin{Corollary}\label{cor:even} For a line graph $G$ not isomorphic to $K_3$, $E_1$, $E_2$ or $E_3$, whenever two triangles in $G$ share an edge, at least one of them is odd.
\end{Corollary}

Beineke gave the characterization in terms of forbidden subgraphs for line graphs in \cite{Beineke} as follows:
\begin{Theorem}\cite{Beineke}\label{Beineke9} A graph is a line graph if and only if it does not contain any of the nine graphs shown in Figure \ref{Fig1} as an induced subgraph.
\begin{figure}[ht!]
\centering
\includegraphics[width=10cm]{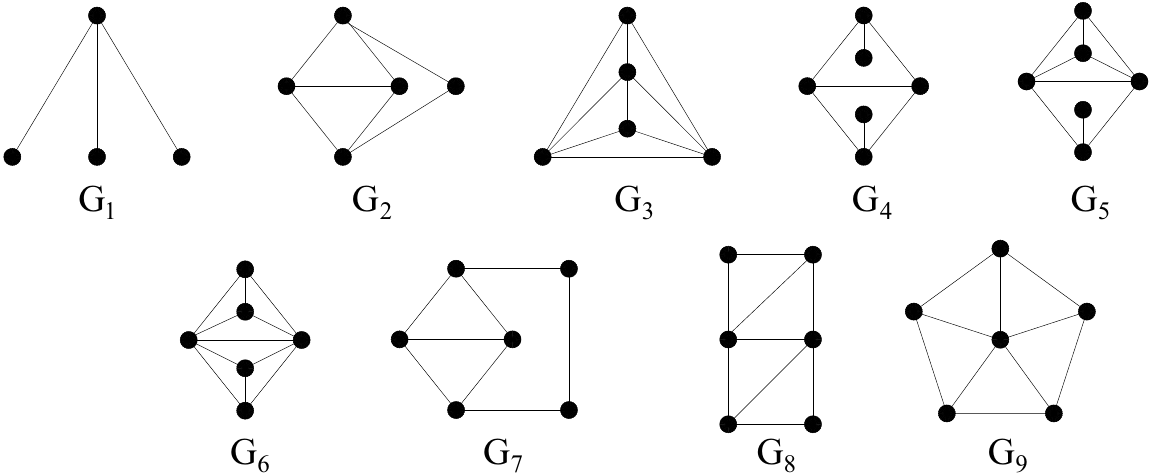}
\caption{The nine forbidden induced subgraphs in line graphs.}
\label{Fig1}
\end{figure}
\end{Theorem}

$\check{\mbox{S}}$olt\'es \cite{Soltes} decreased the number of forbidden subgraphs from nine to seven and proved the following result.

\begin{Theorem}\cite{Soltes}\label{line_forb9} The following five statements are equivalent for a connected graph $G$:
\begin{enumerate}[\rm (a)]
\item $G$ is a line graph.
\item $G$ does not contain any of the graphs $G_1-G_9$ (see Figure \ref{Fig1}) as an induced subgraph.
\item $G$ does not contain any of the graphs $G_1-G_8$ as an induced subgraph and $G$ is not $G_9$.
\item $G$ does not contain any of the graphs $G_1-G_7$ and $G_9$ as an induced subgraph and $G$ is neither $G_8$ nor $H_1$ (see Figure \ref{Fig2}).
\item $G$ does not contain any of the graphs $G_1-G_7$ as an induced subgraph and $G$ is not isomorphic to any of the graphs $G_8, G_9, H_1, H_2$ and $H_3$, (see Figure \ref{Fig2}).
\end{enumerate}
Moreover, for every integer $i < 7$, there are infinitely many connected non-line graphs containing only $G_i$ as an induced subgraph among $G_1-G_9$.
\begin{figure}[ht!]
\centering
\includegraphics[width=9cm]{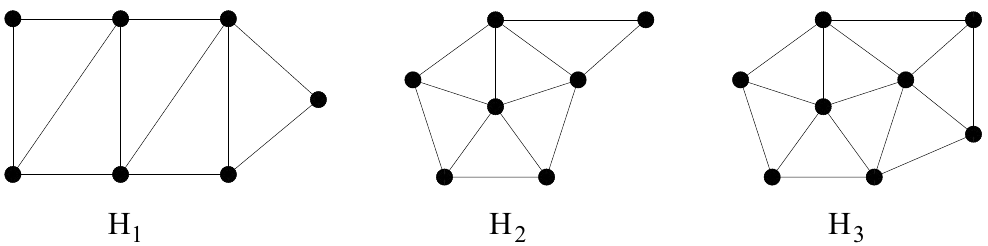}
\caption{Three forbidden induced subgraphs in line graphs.}
\label{Fig2}
\end{figure}
\end{Theorem}

\begin{Observation}\label{SubInducedSub} Let $G$ be a line graph such that $G = L(H)$. Then the following hold.
\begin{enumerate}[\rm (a)]
\item Every induced subgraph of a line graph is itself a line graph.
\item If $G$ is a line graph of $H$ and $G_1$ is an induced subgraph of $G$ then $L^{-1}(G_1) = H_1$ is a subgraph of $H$ not necessarily induced.
\item For an odd induced triangle $T$ in $G$, the edges corresponding to vertices of $T$ in $H$ constitute a star.
\end{enumerate}
\end{Observation}

\begin{Theorem}\label{eventri} For an even induced triangle $T$ in $G \not\in \{K_3, E_1, E_2, E_3\}$, the edges corresponding to vertices of $T$ in $L^{-1}(G)$ constitute a triangle.
\end{Theorem}
\begin{proof}
Assume for the sake of contradiction that the edges in $L^{-1}(G)$, corresponding to the vertices of the even triangle $T$ in $G$, form a star $K_{1,3}$.

Now, as the triangle $T$ is even, any edge connected to the star in $L^{-1}(G)$ is between two of the vertices of the star; otherwise there is an edge from one of the vertices of star to a vertex outside star in $L^{-1}(G)$, the corresponding vertex in $G$ will be connected to exactly one or three vertices of the triangle $T$ which makes it odd - a contradiction. This implies there can not be more than $4$ vertices in $L^{-1}(G)$; otherwise  the star will be disconnected from those vertices. Hence, the only possible cases are $K_{1,3}$, $K_{1,3}+e$, $K_4 - e$ and $K_4$. The $L(G)$ for these are $K_3$, $E_1$, $E_2$ and $E_3$, which leads to contradiction. This completes the proof.
\end{proof}

\begin{Corollary}\label{Corr2} Let $G$ be a line graph and $G \neq K_3$. Then every even triangle in $G$ shares an edge with at least one other triangle.
\end{Corollary}

ROUSSOPOULOS \cite{Rouss} gave a $\max\{m,n\}$ algorithm for determining the graph $H$ from its line graph $G$, where $m$ and $n$ are the number of edges and vertices, respectively. This algorithm also determines whether the given graph $G$ is a line graph or not. For a given line graph $G$, $H = L^{-1}(G)$ is constructed as follows: The vertices of $H$ correspond to the sets of $P$ together with the set of $W$ of vertices of $G$ belonging to only one of the cells of $P$, where $P$ is the set of vertices corresponding to cells of line partition of $G$. Two vertices in $H$ are adjacent whenever the corresponding sets have a non empty intersection.    

\section{Main results}

In this section, we continue to use the terms and notations defined in the previous section. We first prove the following result, which we use to prove the next result about forbidden subgraphs in second  order line graphs. 

\begin{Theorem} Let $G \neq K_3$ be a line graph of $H$ and $G'$ be an induced subgraph of $G$ such that $G' = L(H')$. Then $H'$ is an induced subgraph of $H$ if and only if there is no $v \in V(G) \setminus V(G')$ such that $v$ is adjacent to two vertices of $G'$ and both edges which join $v$ to two vertices of $G'$ are in different cells of the line  partition of subgraph induced by $G' \cup \{v\}$.
\end{Theorem}
\begin{proof} \textsf{Necessity:} Suppose $H'$ is an induced subgraph of $H$. If possible, assume that there is $v \in V(G) \setminus V(G')$ such that $v$ is adjacent to two vertices of $G'$ such that both edges which joins $v$ to $G'$ are in different cells of the line partition of subgraph induced by $G' \cup \{v\}$. Let $L(H'') = G' \cup \{v\}$ and $v$ is adjacent to $v_1$ and $v_2$. Using Algorithm given in \cite{Rouss}, it is clear that in $H'$, we have $v_1,v_2 \in W$ and hence the vertices corresponding to $v_1$ and $v_2$ are pending vertices while in $H''$, we have $\{v,v_1\},\{v,v_2\} \in P$ and hence the vertices corresponding to $\{v,v_1\}$ and $\{v,v_2\}$ are adjacent in $H$. Hence, we obtain $H'$ is not an induced subgraph of $H$ which is a contradiction. Therefore, our assumption is not true which completes the proof of necessity part.

\textsf{Sufficiency:} Suppose $G' = L(H')$ is an induced subgraph of $G$ and there is no $v \in V(G) \setminus V(G')$ such that $v$ is adjacent to two vertices of $G'$ and both edges which join $v$ to two vertices of triangle of $G'$ are in different cells of the line  partition of subgraph induced by $G' \cup \{v\}$. If possible, assume that $H'$ is not an induced subgraph of $H$. Since $H'$ is not an induced subgraph $H$, there exist two vertices $v_1$ and $v_2$ in $H'$ such that $v_1$ and $v_2$ are adjacent in $H$ but not in $H'$. Let $e$ be the edge in $H$ which joins $v_1$ and $v_2$. Since $G$ is connected graph, note that these two vertices $v_1$ and $v_2$ are incident to $x$ and $y$, respectively (possibly $x=y$). Let $v_1x = e_1$ and $v_2y = e_2$.  Note that $e \in V(L(H))$ but $e \not\in V(L(H')) = V(G')$. Hence, $e \in V(G) \setminus V(G')$ such that $e$ is adjacent to two vertices $e_1 \in V(G')$ and $e_2 \in V(G')$ of $L(H) = G$ and note that they are in different cells as the vertices $\{v_1\}$ and $\{v_2\}$ corresponding $e_1$ and $e_2$ are different, which is a contradiction. This completes the proof.
\end{proof}

A subgraph $G'$ of a line graph $G = L(H)$ is called a \emph{pure induced subgraph} if $G' = L(H')$ then $H'$ is an induced subgraph of $H$. 

\begin{Theorem}\label{line:L2NS} Let $G$ be any connected graph and $G$ is not isomorphic to any of the graphs $G_8,G_9,H_1,H_2$ and $H_3$. Then $G = L^2(H)$ for some graph $H$ if and only if the following holds:
\begin{enumerate}[\rm (a)]
\item Every triangle in $G$ is a part of some induced subgraph isomorphic to $K_4-e$.
\item $G$ does not contain any of the graphs $G_2-G_7$ as an induced subgraph and $L(G_2)-L(G_7)$ as a pure induced subgraph.
\end{enumerate}
\end{Theorem}
\begin{proof}\textsf{Necessity:} Suppose $G = L^2(H)$ for some graph $H$. Then, by Theorem \ref{whitney}, it is clear that (b) is satisfied. We now prove (a). Since $L^{-1}(G)$ is claw free, it is clear that $G$ has no triangle alone. Moreover, if $L^{-1}(G)$ contains triangle and as $G$ is connected, we have every triangle is a part of $K_4-e$.

\textsf{Sufficiency:} Suppose (a) and (b) hold for $G$. Since $G$ does not contain $G_1-G_7$ as an induced subgraph, by Theorem \ref{line_forb9} we obtain $G = L(H')$ for some graph $H'$. Since $G$ does not contain $L(G_1)-L(G_7)$ as pure induced subgraphs, $L^{-1}(G) = H'$ does not contain $G_1-G_7$ as induced subgraphs and hence $H' = L(H)$ by Theorem \ref{line_forb9}. Hence, we obtain, $G = L(H') = L(L(H)) = L^2(H)$ for some graph $H$.
\end{proof}

\begin{Theorem}\label{line:L2} Let $G$ be any connected graph and $G$ is not isomorphic to any of the graphs $K_3,G_8,G_9,H_1,H_2$ and $H_3$. Then $G = L^2(H)$ for some graph $H$ if the following holds:
\begin{enumerate}[\rm (a)]
\item Every triangle in $G$ is a part of some induced subgraph isomorphic to $K_4-e$.
\item $G$ does not contain any of the graphs $G_1-G_7$ (given in Figure \ref{Fig1}) and $F_1-F_2$ (given below in Figure \ref{Fig:F}) as an induced subgraph.
\begin{figure}[ht!]
\centering
\includegraphics[width=9cm]{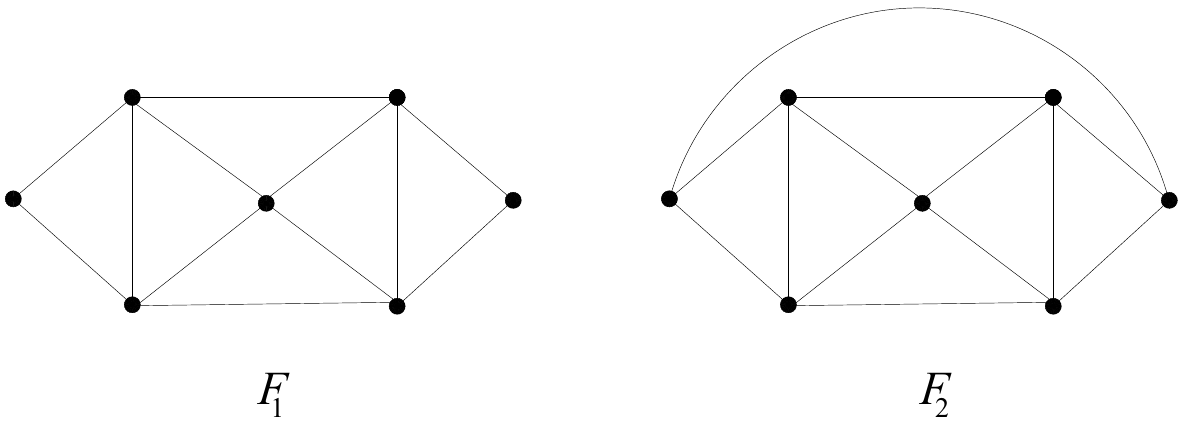}
\caption{Graphs $F_1=L(G_4)$ (left) and $F_2=L(G_2)$ (right).}\label{Fig:F}
\end{figure}
\end{enumerate}
\end{Theorem}
\begin{proof} Let $G$ be a connected graph that is not isomorphic to any of the graphs $K_3$, $G_8$, $G_9$, $H_1$, $H_2$, $H_3$, and $G$ satisfy the conditions (a) and (b) of the hypotheses. We prove that $G = L^2(H)$ for some graph $H$.

Note that, by Theorem \ref{line_forb9} (e), it suffices to prove that $L^{-1}(G)$ does not contain any of the graphs $G_1-G_7$ as an induced subgraph as this ensures the existence of a graph $H$ such that $L^{-1}(G)=L(H)$. Since $G$ does not contain $F_1=L(G_4)$ and  $F_2=L(G_2)$ as an induced subgraph, $L^{-1}(G)$ does not contain $G_4$ and $G_2$ as a subgraph. Since $G$ does not contain $F_1$ as an induced subgraph, $L^{-1}(G)$ does not contain $G_4$ as a subgraph of $G$ and hence $G$ does not contain $G_5$, $G_6$ and $G_7$ as an induced subgraph as $G_4$ is a subgraph of $G_5$, $G_6$ and $G_7$. Similarly, $G$ does not contain $F_2$ as an induced subgraph implies $L^{-1}(G)$ does not contain $G_2$ as a subgraph. Hence $G$ does not contain $G_3$ as an induced subgraph because $G_3$ contains $G_2$ as a subgraph. Now it remains to show that $L^{-1}(G)$ does not contain $G_1 = K_{1,3}$.

Suppose $L^{-1}(G)$ contains $G_1 = K_{1,3}$ as an induced subgraph. Let $V(K_{1,3}) = \{a,b,c,d\}$ with $E(K_{1,3}) = \{ab,ac,ad\}$. Since $K_{1,3}$ is an induced subgraph of $L^{-1}(G)$, it is clear that no two vertices of $\{b,c,d\}$ are mutually adjacent. Now, as there is no edge that is connected to exactly two of the edges of $K_{1,3}$, in $L(H)$, there is no vertex connected to exactly two of the vertices of triangle formed by vertices corresponding to the edges $ab, ac$ and $ad$. This implies the triangle is not a part of $K_4-e$ in $G$, which is a contradiction with (a). Hence, $L^{-1}(G)$ does not contain $K_{1,3}$ which completes the proof.
\end{proof}

\begin{Theorem}\label{line:L2N} Let $G$ be any connected graph and $G$ is not isomorphic to any of the graphs $K_3,G_8,G_9,H_1,H_2$ and $H_3$. If $G = L^3(H)$ for some graph $H$, then every odd triangle in $G$ is a part of some induced subgraph isomorphic to $K_4-e$ and every even triangle in $G$ is a part of some induced subgraph isomorphic to $L(K_4-e)$.
\end{Theorem}
\begin{proof} Since $G = L^3(H) = L^2(L(H))$ for some graph $H$, the graph follows the condition stated in Theorem \ref{line:L2NS} (a), implying every triangle, including odd triangle, is part of an induced subgraph isomorphic to $K_4-e$. Now, consider an even triangle in graph $G$. Let the vertices of the even triangle be $a, b, c$. Now, by Theorem \ref{eventri}, in $L^{-1}(G)$, the corresponding edges to $a, b, c$ will be a part of triangle. Now, as $L^{-1}(G)=L^2(H)$, by Theorem \ref{line:L2NS} (a), the triangle is a part of an induced subgraph isomorphic to $K_4-e$, say $A$. So, $L(A)=L(K_4-e)$, which is an induced subgraph of $L^3(H)$. Therefore, every even triangle in $G$ is a part of some induced subgraph isomorphic to $L(K_4-e)$.
\end{proof}

\begin{Theorem}\label{lin:higher} Let $G$ be any connected graph and $G$ is not isomorphic to any of the graphs $K_3,G_8,G_9,H_1,H_2$ and $H_3$. Then $G = L^n(H), n \geq 2$ for some graph $H$ if the following holds:
\begin{enumerate}[\rm (a)]
\item $L^{-i}(G)$ is claw free for all $i \in \{1,2,....,n-1\}$
\item $G$ does not contain any of the graphs $G_1-G_7$ (given in Figure \ref{Fig1}) and $F_1-F_2$ (given in Figure \ref{Fig:F}) as an induced subgraph.
\end{enumerate}
\end{Theorem}
\begin{proof} We give proof using induction. Now, the statement holds for $n=2$ as proved in Theorem \ref{line:L2}.

For any graph $G$, let the statement hold for $k \in \{3,...,n-1\}$. We now prove that the statement holds for $k=n$. Let a graph $G$ satisfy conditions (a) and (b) stated in the hypothesis. Now, as graph $G$ does not contain $G_1-G_7$, $G$ is a line graph by Theorem \ref{line_forb9}. Since $G$ does not contain $F_1-F_2$, $L^{-1}(G)$ does not contain $G_2-G_7$ as subgraphs by observation \ref{SubInducedSub} and hence $L^{-1}(G)$ does not contain $G_2-G_7$ as an induced subgraph. As $L(F_1)$ and $L(F_2)$ contain $F_1$ as an induced subgraph, $G$ does not contain $L(F_1)$ and $L(F_2)$ as an induced subgraph. Hence, $L^{-1}(G)$ does not contain $F_1$ and $F_2$ as subgraphs by observation \ref{SubInducedSub}. Also, $L^{-i}(G)$ is claw free for all $i \in \{1,2,...,k-1\}$. So, by induction hypothesis, $L^{-1}(G)=L^{n-1}(H)$ for some graph $H$. Therefore, $G=L^{n}(H)$. Therefore, by induction, the theorem is true for all $n$.
\end{proof}

\begin{Theorem}\label{LD3} Let $G$ be a connected graph with $\Delta(G) = 3$ and $G \neq G_2$. Then
\begin{enumerate}[\rm (a)]
\item $G = L(H)$ for some graph $H$ if and only if $G$ does not contain any of $G_1, G_4, G_7$ (shown in Figure 2) as an induced subgraph.
\item $G = L^2(H)$ for some graph $H$ if and only if $G$ is isomorphic to $L_{k,n}$, where $k \in \{1,2\}$ and $n \geq k$ as shown in Figure \ref{Fig:Types}.
\begin{figure}[ht!]
\centering
\includegraphics[width=9cm]{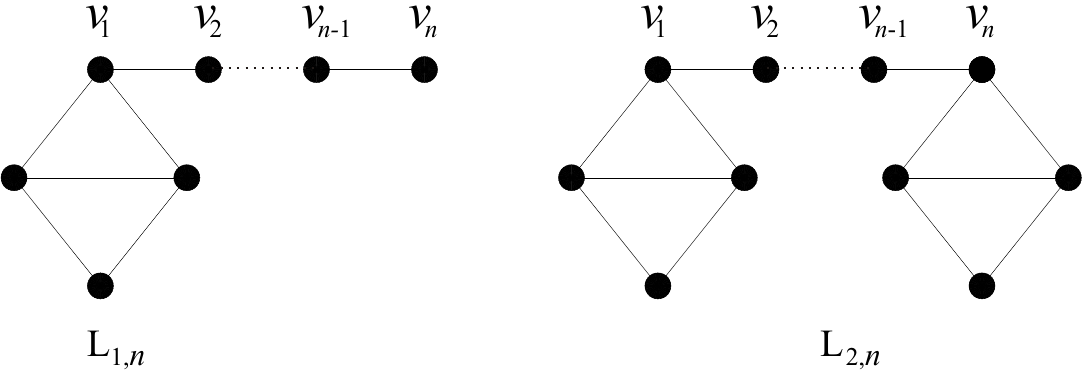}
\caption{Two classes of line graphs $G = L^2(H)$ with $\Delta(G) = 3$.}\label{Fig:Types}
\end{figure}
\item For $n \geq 3$, $G \neq L^n(H)$ for any graph $H$.
\end{enumerate}
\end{Theorem}
\begin{proof} (a) \textsf{Necessity:} Suppose $G$ is a line graph with $\Delta(G) = 3$ and $G \neq G_2$. By Theorem \ref{Beineke9}, it is clear that $G$ does not contain any of the $G_1,G_4,G_7$ as an induced subgraph.

\textsf{Sufficiency:} Suppose $G$ is not $G_2$ and $G$ does not contain any of $G_1,G_4,G_7$ as an induced subgraph. Since $\Delta(G_3) = \Delta(G_5) = \Delta(G_8) = 4$ and $\Delta(G_6) = \Delta(G_9) = 5$, $G$ does not contain any of $G_3,G_5,G_6,G_8,G_9$ as an induced subgraph. We now prove that $G$ does not contain $G_2$ as an induced subgraph. Suppose $G$ contains $G_2$ as an induced subgraph. Since $G \neq G_2$, there is a vertex in $G$ which is adjacent to degree $2$ vertex of $G_2$ and not adjacent to any other vertex of $G_2$. Then observe that $G$ contains $K_{1,3}$ as an induced subgraph which is a contradiction. Hence $G$ does not contain $G_2$ as an induced subgraph. By Theorem \ref{Beineke9}, $G$ is a line graph.

(b) \textsf{Necessity:} Since $G$ is simple and $\Delta(G) = 3$, it is clear that $|V(G)| \geq 4$ and hence $G \neq K_3$. Moreover, observe that $G$ does not contain $K_n$ for $n \geq 5$ as an induced subgraph; otherwise $\Delta(G) \geq 4$. Also, $G$ does not contain $K_4$; otherwise $L^{-1}(G)$ contains $K_{1,3}$ which is a contradiction and there are no possible connections with vertices of $K_4$ as $\Delta(G)=3$. Since $\Delta(G) = 3$, let $u$ be the vertex with $d(u) = 3$. Assume that $u$ is adjacent to $v_1, w$ and $x$. Since $G$ does not contain  $K_{1,3}$ as an induced subgraph, $u$ forms a triangle with two of $\{v_1,w,x\}$. Assume $u,w$ and $x$ forms a triangle. By Theorem \ref{line:L2NS}, note that every triangle is a part of $K_4 - e$. Now we have the following two cases.

\textsf{Case-1:} $v_1$ is adjacent to $w$. Observe that $u$ and $w$ are not adjacent to any other vertex of $G$ as $\Delta(G) = 3$ and $d(u) = d(w) = 3$. Moreover, note that $v_1$ and $x$ are not adjacent; otherwise $G$ is isomorphic to $K_4$ (outer connections are not possible as $\Delta(G)=3$) which is a contradiction.

\textsf{Claim-1:} At least one of $\{v_1,x\}$ has degree $2$.

If possible, then assume that $d(v_1) = d(x) = 3$. In this case, we have the following possibilities: (i) Let $v_2$ be the vertex adjacent to both $v_1$ and $x$ then note that $G$ contains $G_2$ - a contradiction with $G$ being a line graph, (ii) Let $v_2$ and $y$ be two vertices adjacent to $v_1$ and $x$, respectively. If $v_2$ and $y$ are adjacent then $G$ contains $G_7$ - a contradiction with $G$ being a line graph. If $v_2$ and $y$ are not adjacent then $G$ contains $G_4$ as an induced subgraph - a contradiction with $G$ being a line graph. Hence, our assumption $d(v_1) = d(x) = 3$ is not true. Therefore, at least one of $v_1$ and $x$ has degree $2$.

Without loss of generality, assume that $d(x) = 2$. If $d(v_1) = 2$ then $G$ is $K_4 - e = L_{1,1}$. If $d(v_1) = 3$ then let $v_2$ is adjacent to $v_1$. If $d(v_2) = 1$ then $G$ is $L_{1,2}$. Continuing in this way, let $v_2,v_3,\ldots,v_n$ be the sequence of vertices such that $v_i$ is adjacent to $v_{i-1}$ for $2 \leq i \leq n$ and $d(v_n) = 1$, $d(v_i) = 2$ for $2 \leq i \leq n-1$ then $G$ is $L_{1,n}$. Let $v_n$ be the first vertex such that $d(v_n) = 3$. Since $v_n$ is the first vertex in the sequence, $v_n$ is not adjacent to any of $\{v_1.v_2,\ldots,v_{n-2},u,w,x\}$. Let $v_n$ is adjacent to $a$ and $b$. Again, note that $a$ and $b$ are not adjacent to $\{v_1,v_2,\ldots,v_{n-1},u,w,x\}$. Since $G$ is $K_{1,3}$-free, $a$ and $b$ are adjacent and hence $v_n, a$ and $b$ form a triangle. By Theorem \ref{line:L2NS}, note that $v_nab$ is a part of $K_4-e$. Let $c$ is adjacent to $a$ and $b$ but observe that it is not adjacent to $v_n$ as $G$ does not contain $K_4$. By claim-1, it is clear that $d(c) = 2$ and hence $G$ is $L_{2,n}$.

\textsf{Case-2:} $v_1$ is not adjacent to $w$. Then let $y$ be the vertex adjacent to both $w$ and $x$. By claim-1, note that $d(y) = 2$. Now, the subgraph induced by $v_1, w, x, y, u$ is isomorphic to $L_{1,2}$. Now continuing similar to case-1, it is easy to show that $G$ is either $L_{1,n}$ or $L_{2,n}$ for $n \geq 1$.

\textsf{Sufficiency:} Suppose $G$ is isomorphic to $L_{k,n}$, where $k \in \{1,2\}$ and $n \geq k$, then note that $G = L^2(H)$, where $H$ is the graph obtained by identifying a vertex of a triangle with one end vertex of path $P_n, n \geq 2$.  

(c) Suppose there exists a graph $G$ with $\Delta(G) = 3$ such that $G = L^n(H)$ for some $n \geq 3$. In this case, observe that $G$ has an induced subgraph $K$ such that $L^{-(n-2)}(K) = K_4-e$ by Theorem \ref{line:L2NS} (a), $L^2(H)$ contains $K_4-e$ as an induced subgraph. Then note that $\Delta(L(K_4-e)) = 4$ which is a contradiction with $\Delta(G) = 3$. Therefore, we obtain that for $G \neq L^4(H)$ for any graph $H$. Since every $L^n(H)=L^3(L^{n-3}(H))$, we obtain that $G \neq L^n(H)$ for $n \geq 3$.
\end{proof}

Let $G$ be a simple finite connected graph. A subgraph $H$ of a graph $G$ is called a pendant subgraph if $H$ is connected to $G-H$ by an edge only.

\begin{Theorem} Let $G$ be a connected graph with $\Delta(G) = 4$ and $G \neq G_3,G_8,G_9,H_1$, $H_2,H_3$. Then
\begin{enumerate}[\rm (a)]
\item $G = L(H)$ for some graph $H$ if and only if $G$ does not contain any of $G_1,G_2,G_4,\ldots,G_7$ as an induced subgraph and $G$ does not contain $G_3$ as a pendant subgraph.
\item $G = L^3(H)$ for some graph $H$ if and only if $G$ is isomorphic to
$L(L_{k,n})$, where $k \in \{1,2\}$ and $n \geq k$.
\item For $n \geq 4$, $G\neq L^n(H)$ for any graph $H$.
\end{enumerate}
\end{Theorem}
\begin{proof} (a) \textsf{Necessity:} Suppose $G$ is a line graph with $\Delta(G) = 4$ and $G \neq G_3, G_8, G_9, H_1, H_2, H_3$. By Theorem \ref{line_forb9}, $G$ does not contain any of the $G_1,G_2,G_3$, $\ldots,G_7$. Since $G$ does not contain $G_3$, $G$ does not contain $G_3$ as a pendant subgraph.

\textsf{Sufficiency:} Suppose $G$ is a connected graph with $\Delta(G) = 4$ and $G \not\in \{G_3,G_8$, $G_9,H_1,H_2,H_3\}$ and, $G$ does not contain any of $G_1,G_2,G_4,\ldots,G_7$ and $G$ does not contain $G_3$ as a pendant subgraph. It is enough to prove that $G$ does not contain $G_3$ with two edge connections in $G-G_3$. If possible, assume that $G$ contains $G_3$ with two edges in $G-G_3$. Let $\{a,b,c,d,x\}$ induces $G_3$ in $G$ such that $abc$ and $bcd$ are two triangles and $x$ is adjacent to $a, b, c$ and $d$. Observe that $d(b) = d(c) = d(x) = 4$ and hence $b, c$ and $x$ are not adjacent to any other vertex of $G$ as $\Delta(G) = 4$. Moreover, observe that $a$ and $d$ are not adjacent; otherwise, the subgraph induced by $\{a,b,c,d,x\}$ is $K_4$ and in this case, $G$ is $K_4$ only and hence $\Delta(G) = 3$ which is a contradiction. Since $G \neq G_3$, $d(a) = d(d) = 3$ and $\Delta(G) = 4$, there exists either one vertex or two distinct vertices that are adjacent to $a$ and $d$. Then we have the following possible cases and in each case, we obtain that $G$ contains one of graph $\{G_2,G_4,G_7\}$ as an induced subgraph which is a contradiction with $G$ being a line graph:
\begin{itemize}
\item If a vertex $u$ is adjacent to both $a$ and $d$ then the subgraph induced by $\{a,b,d,x,u\}$ is $G_2$ and hence $G$ contains $G_2$ as an induced subgraph.
\item If two vertices $u$ and $v$ are adjacent to $a$ and $d$, respectively, then 
\begin{itemize}
\item if $u$ and $v$ are adjacent then the subgraph induced by $\{a,b,d,x,u,v\}$ is $G_7$ and hence $G$ contains $G_7$ as an induced subgraph.
\item if $u$ and $v$ are not adjacent then the subgraph induced by $\{a,b,d,x,u,v\}$ is $G_4$ and hence $G$ contains $G_4$ as an induced subgraph.    
\end{itemize} 
\end{itemize}
Therefore, in all above cases, our assumption that $G$ contains $G_3$ with at least two edge connections in $G$ leads to a contradiction. Hence, $G$ does not contain $G_3$ as a non-pendant subgraph.

(b) \textsf{Necessity:} Let $G$ be a connected graph with $\Delta(G) = 4$ and $G=L^3(H)$ for some connected graph $H$. Note that $L^{-i}(G),\;i=1,2$ does not contain $K_{1,3}$ as an induced subgraph as $G = L^3(H) = L(L^2(H)) = L(L(L(H)))$. Since $\Delta(G) = 4$, there exists a vertex $u$ such that $d(u) = 4$. Let $u_1,u_2,u_3$ and $u_4$ be the vertices adjacent to $u$. As $G=L(L^2(H))$, $u$ will form a triangle with two of $u_1, u_2, u_3, u_4$ as it does not contain a $K_{1,3}$ as an induced subgraph by Theorem \ref{Beineke9}. By Theorem \ref{line:L2NS}, every triangle is a part of $K_4 - e$ and by Theorem \ref{rooij}, observe that one of those triangles is even. Hence, by Theorem \ref{line:L2N}, even triangle is a part of $E_2 = L(K_4-e)$, which is an induced subgraph of $G$. Let $u_1-u_2-u_3-u_4-u_1$ be a cycle and $u$ be the central vertex of degree $4$ in the induced graph $E_2$.

\textsf{Claim-1:} If $G \neq E_2$ then there is a vertex $v_1 \in V(G) \setminus V(E_2)$ such that $v_1$ is adjacent to two consecutive vertices of cycle form by $\{u_1,u_2,u_3,u_4\}$.

Let $v_1 \in V(G) \setminus V(E_2)$. It is clear that $v_1$ is not adjacent to $u$ as $d(u) = 4$ and $\Delta(G) = 4$. We consider the following possibilities:
\begin{itemize}
\item If $v_1$ is adjacent to one vertex only, say $u_1$, then the subgraph induced by $\{u_1,v_1,u_2,u_4\}$ is $K_{1,3}$ centered at $u_1$ which is a contradiction with $G$ being a line graph.
\item If $v_1$ is adjacent to two non-consecutive vertices of the  cycle of $E_2$, then the subgraph induced by $\{v_1,u,u_1,u_2,u_3\}$ is $G_2$ which is a contradiction with $G$ being a line graph.
\item If $v_1$ is adjacent to three vertices, say $u_1,u_2,u_3$, of the cycle of $E_2$, then the subgraph induced by $\{v_1,u,u_1,u_3,u_4\}$ is $G_2$ which is a contradiction with $G$ being a line graph.
\item If $v_1$ is adjacent to all four vertices $u_1,u_2,u_3,u_4$ then $V(G) = \{v_1,u,u_1,\ldots,u_4\}$ and $G = L(K_4)$ and hence $L^{-2}(G) = K_{1,4}$ which contains $K_{1,3}$ which is a contradiction with $L^{-2}(G) = L(H)$.
\end{itemize}
Hence, $v_1 \in V(G) \setminus V(E_2)$ is adjacent to two consecutive vertices, say $u_1$ and $u_2$, of cycle induced by $\{u_1,u_2,u_3,u_4\}$ which completes the proof of claim-1.

\textsf{Claim-2:} There is no $x \in V(G) \setminus \{v_1,u,u_1,\ldots,u_4\}$ such that $x$ is adjacent to any of $\{u,u_1,\ldots,u_4\}$.

If possible, assume that $x$ is adjacent to some vertex from $\{u,u_1,\ldots,u_4\}$. Since $d(u) = d(u_1) = d(u_2) = 4$, $x$ is not adjacent to $u$, $u_1$ and $u_2$. If $x$ is adjacent to $u_3$ only, then the subgraph induced by $\{v_1,x,u_1,u_3,u_4\}$ is $G_4$ which is a contradiction with $G$ being a line graph. If $x$ is adjacent to both $u_3$ and $u_4$ then $L^{-1}(G)$ contains 
\begin{itemize}
\item $G_4$ if $x$ and $v_1$ are not adjacent and no vertex adjacent to both $x$ and $v_1$,
\item $G_2$ if $x$ and $v_1$ are adjacent and no vertex adjacent to both $x$ and $v_1$,
\item $G_7$ if another vertex $y$ is adjacent to both non-adjacent vertices $v_1$ and $x$, 
\item if a vertex $y$ is adjacent to both adjacent vertices $v_1$ and $x$, then it also has to be adjacent to all of $u_1, u_2, u_3$ and $u_4$.
\end{itemize}

Therefore, in all above cases we obtain a contradiction as $L^{-1}(G) = L(L(H))$ is a line graph and $\Delta(G)=4$. The proof of claim-2 is complete.

\textsf{Claim-3:} $v_1-v_2-\ldots-v_n (n \geq 2)$ is a path then either $d(v_n) = 1$ or $v_n$ is adjacent to two consecutive vertices of the $4$-cycle of $E_2$.

Assume that $d(v_n) \neq 1$. We prove that $v_n$ is a part of $E_2$. It is clear that $d(v_n) \leq 4$ as $\Delta(G) = 4$. Note that $v_n$ is the first vertex on the path joined to $v_1$ such that $d(v_n) \geq 3$. If possible, assume that $d(v_n) = 4$ then as mentioned in the construction earlier, there is a vertex $x$ such that $d(x) = 3$ and it appears prior to $v_n$ on a path joining $v_1$ and $v_n$ which is a contradiction. Hence, we obtain $d(v_n) \neq 4$. Assume that $d(v_n) = 3$. Let $v_n$ be adjacent to two other vertices $w_1$ and $w_2$. Since $v_n$ is the first vertex with $d(v_n) = 3$ and $G$ does not contain $K_{1,3}$, $w_1$ and $w_2$ are adjacent. Hence, $v_n,w_1$ and $w_2$ form a triangle. By Theorem \ref{line:L2N}, triangle $v_n,w_1,w_2$ is a part of $K_4-e$. Now repeating the argument as earlier again, observe that the graph attached to $v_n$ is same as one at $v_1$.  

It is clear from claims 1 to 3 that $G$ is one of the graphs $L(L_{1,n})$ or $L(L_{2,n})$, where $n \geq 1$.

\textsf{Sufficiency:} Suppose $G$ is isomorphic to $L(L_{k,n})$, where $k \in \{1,2\}$ and $n \geq k$, then note that $G = L^3(H)$, where $H$ is the graph obtained by identifying a vertex of a triangle with one end vertex of path $P_n\,\;n \geq 2$. 

(c) It is clear from (a) that if $G = L^3(H)$ and $\Delta(G) = 4$, then either $G$ is $L(L_{1,n})$ or $L(L_{2,n})$. Since $\Delta(L^{n-1}(L(L_{i,n}))) \geq 5$ for $i \in \{1,2\}$ and $n \geq 4$, $G \neq L^n(H)$ for and graph $H$.
\end{proof}


\begin{thebibliography}{00}
\bibitem{Beineke} Beineke, L. W., Derived graphs and digraphs, Beitrage zur Graphentheorie (Teubner, Leipzig 1968), 17--33 (1968).

\bibitem{Krausz} Krausz, J., Demonstration nouvelle d'une theoröme de Whitney sur les réseaux, Mat. Fiz. Lapok, \textbf{50}, 75-85 (1953).

\bibitem{Lai} Lai, H. and Soltes, L., Line graphs and forbidden induced subgraph, J. of Combi. Theory, Series B, \textbf{82}, 38--55 (2001).

\bibitem{Rooij} van Rooij, A.C.M., Wilf, H.S., The interchange graph of a finite graph, Acta Math. Acad. Sci. Hungar., \textbf{16}, 263--269 (1965).

\bibitem{Rouss} ROUSSOPOULOS, N. D., A $\max\{m,n\}$ algorithm for determining the graph $H$ from its line graph $G$, Information Processing Letters, \textbf{2}, 108--112 (1973).

\bibitem{Soltes} $\check{\mbox{S}}$olt\'es, L., Forbidden induced subgraphs for line graphs, Discrete Mathematics, \textbf{132}, 391--394 (1994).

\bibitem{Whitney} Whitney, Hassler., Congruent Graphs and the Connectivity of Graphs, American Journal of Mathematics. \textbf{54(1)}, 150--168 (1932).
\end{thebibliography}
\end{document}